\documentclass[12pt]{article}

\usepackage{amsmath,amsfonts,amssymb,amsthm,cite}
\usepackage{geometry}

\theoremstyle{plain}
\newtheorem{theorem}{Theorem}
\newtheorem{lemma}{Lemma}

\newtheorem{corollary}{Corollary}
\newtheorem{definition}{Definition}
\theoremstyle{definition}

\newcommand{\nonprint}[1]{}

\begin{document}

\begin{flushleft}

\vspace{+0.3cm}
\large

\begin{center}
\textbf{The solvability of inhomogeneous boundary-value problems in Sobolev spaces}\\
\end{center}

\begin{center}
\vspace{+0.4cm}
\textbf{Vladimir Mikhailets, Olena Atlasiuk}
\vspace{+0.3cm}
\end{center}

\end{flushleft}

\normalsize

\begin{abstract}
The aim of the paper is to develop a general theory of solvability of linear inhomogeneous boundary-value problems for systems of ordinary differential equations of arbitrary order in Sobolev spaces. Boundary conditions are allowed to be overdetermined or underdetermined. They may contain derivatives, of the unknown vector-valued function, whose integer or fractional orders exceed the order of the differential equation. Similar problems arise naturally in various applications. The theory introduces the notion of a rectangular number characteristic matrix of the problem. The index and Fredholm numbers of this matrix coincide respectively with the index and Fredholm numbers of the inhomogeneous boundary-value problem. Unlike the index, the Fredholm numbers (i.e. the dimensions of the problem kernel and co-kernel) are unstable even with respect to small (in the norm) finite-dimensional  perturbations. We  give examples in which the characteristic matrix can be explicitly found. We also prove a limit theorem for a sequence of characteristic matrices. Specifically, it follows from this theorem that the Fredholm numbers of the problems under investigation are semicontinuous in the strong operator topology. Such a property ceases to be valid in the general case.
\end{abstract}



\section{Introduction and statements of the problems}\label{section1}

The study of systems of ordinary differential equations is an important part of many investigations in modern analysis and its applications (see, e.g., \cite{BochSAM2004} and references therein). Unlike Cauchy problems, the solutions to inhomogeneous boundary-value problems for differential systems may not exist and/or may not be unique. Therefore, the question about the solvability character of such problems is fundamental for the theory of differential equations. The question is most fully studied for linear ordinary differential equations. Thus, Kiguradze \cite{Kigyradze1975, Kigyradze1987} and Ashordia \cite{Ashordia} investigated the solutions of first order differential systems with general inhomogeneous boundary conditions of the form
\begin{equation}\label{kig}
y^{\prime}(t) + A(t)y(t)=f(t), \quad t\in(a,b), \quad By=c.
\end{equation}
Here, the matrix-valued function $A(\cdot)$ is Lebesgue integrable over the finite interval $(a, b)$; the vector-valued function $f(\cdot)$ belongs to $L\left((a, b); \mathbb{ R}^m\right)$; the vector $c$ pertains to $\mathbb{R}^{m}$, and $B$ is an arbitrary linear continuous operator from the Banach space $C\left([a,b]; \mathbb {R}^{m}\right)$ to $\mathbb{R}^{m}$, with $m \in \mathbb{N}$. The boundary condition in \eqref{kig} covers the main types of classical boundary conditions; namely: Cauchy problems, two-point and multipoint problems, integral and mixed problems. The Fredholm property with zero index was established for problems of the form \eqref{kig}. Moreover, the conditions for the problems to be well posed were obtained, and the limit theorem for their solutions was proved.

These results were further developed in a series of articles by Mikhailets and his disciples. Specifically, they allow the differential system to have an arbitrary order $r \in \mathbb{N}$ and the boundary operator $B$ to be  any linear continuous operator from the space $C^{r-1}\left ([a,b]; \mathbb{C}^{m}\right)$ to $\mathbb{C}^{rm}$. They obtained  conditions for the boundary-value problems to be well posed and proved limit theorems for solutions to these problems. These results generalize Kiguradze's theorems in the $r=1$ case. Moreover, limit theorems for Green's matrices of such boundary-value problems were established for the first time \cite{KodlyukMikR2013, MikhailetsChekhanova}. These results were applied to the analysis of multipoint boundary-value problems \cite{AtlasiukNelkol2019}, as well as to the spectral theory of differential operators with distributions in coefficients \cite{GoriunovMikhailets2010MN, GoriunovMikhailets2010MFAT, GoriunovMikhailets2012UMJ, GoriunovMikhailetsPankrashkin2013EJDE}.

Note that boundary-value problems with inhomogeneous boundary conditions containing derivatives whose order is greater than or equal to the order of the differential equation naturally arise in some mathematical models (see, e.g., \cite{Kra1961,Luo1991,Vent1959}). The theory of such problems so far contains few results even in the case of ordinary differential equations.

The present article investigates the solvability character of systems of linear ordinary differential equations with the most general inhomogeneous boundary conditions in Sobolev spaces. The boundary conditions may contain derivatives of the unknown functions of integer and fractional orders that exceed the order of the differential system. The systems can be underdetermined or overdetermined.

Let us introduce necessary notation to describe the problem under investigation. Throughout the paper, we arbitrarily choose a finite interval $(a,b)\subset\mathbb{R}$ and the following parameters:
$$
n\in\mathbb{N}\cup\{0\},\quad\{m, r, l\}\subset\mathbb{N},\quad\mbox{and}\quad 1\leq p\leq \infty.
$$
As usual,
\begin{align*}
&W_p^{n+r}\bigl([a,b];\mathbb{C})\\
&:= \bigl\{y\in C^{n+r-1}([a,b];\mathbb{C})\colon y^{(n+r-1)}\in AC[a,b], \, y^{(n+r)}\in L_p[a,b]\bigr\}
\end{align*}
is a complex Sobolev space; set $W_p^0:=L_p$. This space is Banach with respect to the norm
$$
\bigl\|y\bigr\|_{n+r,p}=\sum_{k=0}^{n+r}\bigl\|y^{(k)}\bigr\|_{p},
$$
with $\|\cdot\|_p$ standing for the norm in the Lebesgue space $L_p\bigl([a,b]; \mathbb{C}\bigr)$. We need the Sobolev spaces
$$
(W_p^{n+r})^{m}:=W_p^{n+r}\bigl([a,b];\mathbb{C}^{m}\bigr)
\;\;\mbox{and}\;\;
(W_p^{n+r})^{m\times m}:=W_p^{n+r}\bigl([a,b];\mathbb{C}^{m\times m}\bigr)
$$
They respectively consist of vector-valued functions and matrix-valued functions whose ele\-ments belong to $W_p^{n+r}$. The norms in these spaces are defined to be the sums of the relevant norms (in $W_p^{n+r}$) of all elements of a vector-valued or matrix-valued function. We preserve the same notation $\|\cdot\|_{n+r,p}$ for these norms. It will be clear from the context to which space (scalar or vector-valued or matrix-valued functions) relates the designation of the norm. The same concerns all other Banach spaces used in the sequel. Certainly, the above Sobolev spaces coincide in the $m=1$ case. If $p<\infty$, they are separable and have a Schauder basis.

We consider a linear boundary-value problem of the form
\begin{align}\label{bound_pr_1}
(Ly)(t):=y^{(r)}(t) + \sum\limits_{j=1}^rA_{r-j}(t)y^{(r-j)}(t)&=f(t), \quad t\in(a,b),\\
\label{bound_pr_2}
By&=c.
\end{align}
We suppose that the matrix-valued functions $A_{r-j}(\cdot) \in (W_p^n)^{m\times m}$, vector-valued function $f(\cdot) \in (W^n_p)^m$, vector $c \in \mathbb{C}^{l}$, linear continuous operator
\begin{equation}\label{oper_B_class}
B\colon (W^{n+r}_p)^m\rightarrow\mathbb{C}^{l}
\end{equation}
are arbitrarily chosen and that the vector-valued function $y(\cdot)\in (W_{p}^{n+r})^m$ is unknown.

If $l<r$, then the boundary conditions are underdetermined; if $l>r$, then they are overdetermined.

The boundary condition \eqref{bound_pr_2} consists of $l$ scalar condition for system of $m$ differential equations of $r$-th order, we representing vectors and vector-valued functions as columns. A solution to the boundary-value problem \eqref{bound_pr_1}, \eqref{bound_pr_2} is understood as a vector-valued function $y(\cdot)\in (W_{p}^{n+r})^m$ that satisfies both equation \eqref{bound_pr_1} (everywhere if $n\geq 1$, and almost everywhere if $n=0$) on $(a,b)$ and equality \eqref{bound_pr_2}. If the parameter $n$ increases, so does the class of linear operators \eqref{oper_B_class}. When $n=0$, this class contains all operators that set the general boundary conditions described above.

It is known \cite{Ioffe} that, if $1\leq p < \infty$, then every operator \eqref{oper_B_class} admits a unique analytic representation
\begin{equation}\label{st anal}
By=\sum _{i=0}^{n+r-1} \alpha_{i}\,y^{(i)}(a)+\int_{a}^b \Phi(t)y^{(n+r)}(t){\rm d}t, \quad y(\cdot)\in (W_{p}^{n+r})^{m},
\end{equation}
for certain number matrices $\alpha_{s} \in \mathbb{C}^{l\times m}$ and  matrix-valued function $\Phi(\cdot)\in L_{p'}([a, b]; \mathbb{C}^{l\times m})$; as usual, $1/p + 1/p'=1$. If $p=\infty$, this formula also defines a bounded operator $B\colon (W_{\infty}^{n+r})^{m} \rightarrow \mathbb{C}^{l}$. However, there exist other operators of this class generated by integrals over finitely additive measures. Hence, unlike  $p<\infty$ (\cite{GKM2017},  \cite{Atl2}, \cite{KM2013}), the case of $p=\infty$ contains additional analytical difficulties.

The set of solutions to equation \eqref{bound_pr_1} coincides with the space $(W_{p}^{n+r})^m$ when the right-hand side of the equation runs over $(W_{p}^{n})^m$. Hence, the boundary condition \eqref{bound_pr_2} with an operator of the form \eqref{oper_B_class} is the most general for this equation.

The main purpose of this paper is to prove that the boundary-value problem \eqref{bound_pr_1}, \eqref{bound_pr_2} is Fredholm and to find its Fredholm numbers, i.e. the dimensions of its kernel and co-kernel. Along the way, we find the index of the problem. Note that, unlike the index, the Fredholm numbers are unstable with respect to one-dimensional additive perturbations with an arbitrarily small norm. To find these  numbers, we introduce a rectangular number characteristic matrix $M(L,B)$ of the boundary-value problem and prove that the Fredholm numbers of this matrix coincide with the Fredholm numbers of the problem. Furthermore, we find constructive sufficient conditions for the convergence of the sequence of characteristic matrices to the characteristic matrix of the given problem.

The paper is organized as follows:

Section \ref{section3} states the main results; namely: the formula for the index of the problem, the definition of the characteristic matrix, the formulas for the Fredholm numbers, the limit theorem for the characteristic matrices and some of its consequences.

Section \ref{section4} gives examples of differential systems with constant coefficients and boundary conditions for which the characteristic matrix is written in an explicit form.

Section \ref{section5} provides a proof of the formula for the index of the boundary-value problem.

Section \ref{section6} gives a proof of the theorem on the Fredholm numbers of the problem.

Section \ref{section7} contains proofs of the limit theorems for the characteristic matrices of a sequence of boundary-value problems.

The abstract result presented in the Appendix shows that the established limit theorem is of specific character for the considered class of inhomogeneous boundary-value problems.

\section{Solvability and characteristic matrix}\label{section3}

We now formulate the main results of the paper. They will be proved in Sections~\ref{section5}, \ref{section6}, and \ref{section7}.

We rewrite the inhomogeneous boundary-value problem \eqref{bound_pr_1}, \eqref{bound_pr_2} in the form of a linear operator equation
\[
(L,B)y=(f,c).
\]
Here, $(L,B)$ is a bounded linear operator on the pair of Banach spaces
\begin{equation}\label{(L,B)}
(L,B)\colon (W^{n+r}_p)^m\rightarrow (W^{n}_p)^m\times\mathbb{C}^l,
\end{equation}
which follows from the definition of the Sobolev spaces involved and from the fact that $W_p^n$ is a Banach algebra.

Let $E_{1}$ and $E_{2}$ be Banach spaces. A linear bounded operator $T\colon E_{1}\rightarrow E_{2}$ is called a Fredholm one if its kernel and co-kernel are finite-dimen\-si\-onal. If $T$ is a Fredholm operator, then its range $T(E_{1})$ is closed in $E_{2}$, and its index
$$
\mathrm{ind}\,T:=\dim\ker T-\dim\big(E_{2}/T(E_{1})\big)\in \mathbb{Z}
$$
is finite (see, e.g., \cite[Lemma~19.1.1]{Hermander1985}).

\begin{theorem}\label{th_fredh high}
The bounded linear operator \eqref{(L,B)} is a Fredholm one with index $rm-l$.
\end{theorem}

The proof of Theorem \ref{th_fredh high} uses the well-known theorem on the stability of the index of a linear operator with respect to compact additive perturbations (cf. \cite{AtlMikh2018}).

Theorem \ref{th_fredh high} naturally raises the question of finding the Fredholm numbers of the operator $(L, B)$, i.e. $\operatorname{dim} \operatorname{ker}(L, B)$ and $\operatorname{dim} \operatorname{coker}(L, B)$. This is a quite difficult task because the Fredholm numbers may vary even with arbitrarily small one-dimensional additive perturbations.

To formulate the following result, let us introduce some notation and definitions.

For each number $i \in \{1,\dots, r\}$, we consider the family of matrix Cauchy problems:
\begin{equation}\label{zad kosh1}
Y_i^{(r)}(t)+\sum\limits_{j=1}^rA_{r-j}(t)Y_i^{(r-j)}(t)=O_{m},\quad t\in (a,b),
\end{equation}
with the initial conditions
\begin{equation}\label{zad kosh2}
Y_i^{(j-1)}(a) = \delta_{i,j}I_m,\quad j \in \{1,\dots, r\},
\end{equation}
where $Y_i(\cdot)$ is an unknown $m\times m$ matrix-valued function. As usual, $O_{m}$ stands for the zero $m\times m$ matrix, $I_{m}$ denotes the identity $m\times m$ matrix, and $\delta_{i,j}$ is the Kronecker delta. Each Cauchy problem \eqref{zad kosh1}, \eqref{zad kosh2} has a unique solution $Y_i\in(W_p^{n+r})^{m\times m}$ due to Lemma~\ref{ob Koshi} given in Section~\ref{section5}. Certainly, if $r=1$, we use the designation  $Y(\cdot)$ for $Y_1(\cdot)$.

Let $\left[BY_i\right]$ denote the number $l\times m$ matrix whose $j$-th column is the result of the action of $B$ on the $j$-th column of the matrix-valued function~$Y_i$.

\begin{definition}\label{defin_harm}
A bloc rectangular number matrix
\begin{equation}\label{matrix_BY}
M(L,B):=\big(\left[BY_1\right],\dots,\left[BY_{r}\right]\big) \in \mathbb{C}^{l\times rm}
\end{equation}
is called the characteristic matrix of the inhomogeneous boundary-value problem \eqref{bound_pr_1}, \eqref{bound_pr_2}. Note that this matrix consists of $r$ rectangular block columns $\left[BY_k\right]\in \mathbb{C}^{m\times l}$.
\end{definition}

Here, $mr$ is the number of scalar differential equations of the system \eqref{bound_pr_1}, and $l$ is the number of scalar boundary conditions in \eqref{bound_pr_2}.

\begin{theorem}\label{th dimker}
The dimensions of the kernel and co-kernel of the operator \eqref{(L,B)} are equal to the dimensions of the kernel and co-kernel of the characteristic matrix \eqref{matrix_BY}, resp; i.e.,
\begin{gather}\label{dimker}
\operatorname{dim}\operatorname{ker}(L,B)=
\operatorname{dim}\operatorname{ker}M(L,B),\\
\label{dimcoker1}
\operatorname{dim} \operatorname{coker}(L,B)=\operatorname{dim} \operatorname{coker}M(L,B).
\end{gather}
\end{theorem}

Theorem \ref{th dimker} implies the following necessary and sufficient conditions for the invertibility of \eqref{(L,B)}:

\begin{corollary}\label{invertible}
The operator \eqref{(L,B)} is invertible if and only if $l=rm$ and the square matrix $M(L,B)$ is nonsingular.
\end{corollary}

In the $r=1$ case, Theorem~\ref{th_fredh high} and Corollary~\ref{invertible} are proved in~\cite{AtlMikh2018}. In the case where $l=rm$ and $p<\infty$, Corollary~\ref{invertible} is proved in~\cite{GnypKodMik15}. Theorem~\ref{th dimker} is also new for the systems of first order differential equations.

Together with the problem \eqref{bound_pr_1}, \eqref{bound_pr_2}, we consider a sequence of boun\-da\-ry-value problems
\begin{gather}\label{6.syste}
L(k)y(t,k):= y^{(r)}(t,k)+\sum_{j=1}^{r}A_{r-j}(t,k)y^{(r-j)}(t,k)=f(t,k),\\
B(k)y(\cdot,k)=c(k), \quad t\in (a,b), \quad k\in\mathbb{N},\label{6.kue}
\end{gather}
where the matrix-valued functions $A_{r-j}(\cdot,k)$, the vector-valued functions $f(\cdot,k)$, the vectors $c(k)$, and the linear continuous operators $B(k)$ satisfy the above conditions imposed on the problem \eqref{bound_pr_1}, \eqref{bound_pr_2}. We assume in the sequel that $k \in \mathbb{N}$ and that all asymptotic relations are considered for $k\rightarrow \infty$. The boundary-value problem \eqref{6.syste}, \eqref{6.kue} is also the most general (generic) with respect to the Sobolev space $W^{n+r}_p$.

We associate the sequence of linear continuous operators
\begin{equation}\label{L(e)_B(e)}
(L(k),B(k))\colon(W^{n+r}_p)^{m}\rightarrow (W^{n}_p)^{m}\times\mathbb{C}^{l}
\end{equation}
and the sequence of characteristic matrices
$$
M\big(L(k),B(k)\big):=
\big(\left[B(k)Y_1(k)\right],\dots,\left[B(k)Y_{r}(k)\right]\big) \subset \mathbb{C}^{l\times rm}
$$
with the boundary-value problems \eqref{6.syste}, \eqref{6.kue}.

As usual,
\begin{equation}\label{zb LB11}
\left(L(k),B(k)\right)\xrightarrow{s} \left(L,B\right)
\end{equation}
denotes the strong convergence of the sequence of operators $(L(k),B(k))$ to the operator~$(L,B)$.

Let us now formulate a sufficient condition for the convergence of the sequence of characteristic matrices $M\left(L(k),B(k)\right)$ to the matrix $M\left(L,B\right)$.

\begin{theorem}\label{koef matr}
If the sequence of operators $(L(k),B(k))$ converges strongly to the operator $(L,B)$, then the sequence of characteristic matrices $M\big(L(k),B(k)\big)$ converges to the matrix $M\big(L,B\big)$; i.e.,
\begin{equation*}\label{zb har m}
(L(k),B(k))\xrightarrow{s} (L,B)\;\Longrightarrow\; M(L(k),B(k))\rightarrow M(L,B).
\end{equation*}
\end{theorem}

Theorem \ref{koef matr} implies the sufficient conditions for the upper semicontinuity of the dimensions of the kernel and co-kernel of \eqref{L(e)_B(e)}.

\begin{theorem}\label{ker coker}
If condition \eqref{zb LB11} is satisfied, then the following inequalities hold true for all sufficiently large $k$:
\begin{gather}
\operatorname{dim} \operatorname{ker}(L(k),B(k))\leq \operatorname{dim} \operatorname{ker}(L,B), \label{ner ker} \\
\operatorname{dim} \operatorname{coker}(L(k),B(k))\leq \operatorname{dim} \operatorname{coker}(L,B). \label{ner coker}
\end{gather}
\end{theorem}

Let us consider three significant direct consequences of Theorem~\ref{ker coker}. Suppose that condition \eqref{zb LB11} is satisfied.

\begin{corollary}\label{cor1}
If the operator $(L,B)$ is invertible, then so are the operators $\left(L(k),B(k)\right)$ for all sufficiently large $k$.
\end{corollary}

\begin{corollary}\label{cor2}
If the boundary-value problem \eqref{bound_pr_1}, \eqref{bound_pr_2} has a  solution for arbitrarily chosen right-hand sides, then so do the boundary-value problems \eqref{6.syste}, \eqref{6.kue} for all sufficiently large $k$.
\end{corollary}

\begin{corollary}\label{cor3}
If the homogeneous boundary-value problem \eqref{bound_pr_1}, \eqref{bound_pr_2} has only a trivial solution, then so do the homogeneous problems \eqref{6.syste}, \eqref{6.kue} for all sufficiently large $k$.
\end{corollary}

Note that the conclusion of Theorem \ref{ker coker} and its consequences cease to be valid for arbitrary bounded linear operators between infinite-dimensional Banach spaces (see Appendix).

\section{Examples}\label{section4}

If all coefficients of the differential expression $L$ are constant, then the characteristic matrix of the corresponding boundary-value problem can be explicitly found. In this case, the characteristic matrix is an analytic function of a certain square number matrix and coincides hence with some polynomial of this matrix.

\textit{Example $1$.} Consider a linear one-point boundary-value problem for first order constant-coefficient differential equation
\begin{align}\label{1.6.1t1}
    (Ly)(t):= y'(t)+Ay(t)&=f(t),\quad
t \in(a,b),\\
\label{1.3t1}
By:= \sum _{k=0}^{n-1} \alpha_{k} y^{(k)}(a)&=c.
\end{align}
Here, $A\in\mathbb{C}^{m\times m}$; $f(\cdot)\in(W_{p}^{n})^{m}$; all  $\alpha_{k}\in\mathbb{C}^{l\times m}$; $c\in \mathbb{C}^{l}$, and $y(\cdot)\in (W_{p}^{n+1})^m$. Thus, we have the bounded linear operators \begin{equation*}
B\colon (W_{p}^{n+1})^{m} \rightarrow\mathbb{C}^{l}
\quad\mbox{and}\quad
(L,B)\colon (W^{n+1}_p)^m\rightarrow (W^{n}_p)^m\times\mathbb{C}^l.
\end{equation*}

Let $Y(\cdot)=(y_{i,j})_{i,j=1}^{m}\in (W_p^{n+1})^{m\times m}$ be the unique solution of the linear homogeneous matrix equation with the initial Cauchy condition
\begin{equation*}\label{r31}
  Y'(t)+A Y(t)=O_{m},\quad t\in (a,b), \qquad Y(a)=I_{m},
  \end{equation*}
with $I_m$ denoting the identity $m \times m$ matrix. Hence,
\begin{equation*}
Y(t)= \operatorname{exp}\big(-A(t-a)\big),
\quad\mbox{with}\quad Y(a)=I_{m},
\end{equation*}
and
\begin{equation*}
Y^{(k)}(t)= (-A)^k \operatorname{exp}\big(-A(t-a)\big),
\quad\mbox{with}\quad Y^{(k)}(a) = (-A)^k,
\end{equation*}
whenever $k \in \mathbb{N}$. Recall that
\begin{equation*}
M(L,B)=\left( B \begin{pmatrix}
                                              y_{1,1} \\
                                              \vdots \\
                                              y_{m,1} \\
                                            \end{pmatrix}
,\ldots,
                                    B \begin{pmatrix}
                                              y_{1,m} \\
                                              \vdots \\
                                              y_{m,m} \\
                                            \end{pmatrix}\right) \in \mathbb{C}^{l\times m}.
\end{equation*}
Thus,
$$
M(L,B)=\sum_{k=0}^{n-1}\alpha_{k}(-A)^k
$$
by the definition of $B$.

Theorem \ref{th_fredh high} asserts that
$$
\operatorname{ind} \, (L, B)=\operatorname{ind} \, (M(L, B))= m-l.
$$
Thus, owing to Theorem \ref{th dimker}, we obtain
\begin{equation*}
\operatorname{dim} \operatorname{ker}(L,B)=\operatorname{dim} \operatorname{ker}\left(\sum_{k=0}^{n-1}\alpha_{k}(-A)^k\right)=
m-\operatorname{rank}\left(\sum_{k=0}^{n-1}\alpha_{k}(-A)^k\right)
\end{equation*}
and
\begin{align*}
\operatorname{dim} \operatorname{coker}(L,B)&=-m+l+\operatorname{dim} \operatorname{ker}\left(\sum_{k=0}^{n-1}\alpha_{k}(-A)^k\right) \\
&=l-\operatorname{rank}\left(\sum_{k=0}^{n-1}\alpha_{k}(-A)^k\right).
\end{align*}
It follows from these formulas that the Fredholm numbers of the problem do not depend on the length of the interval $(a, b)$.

\textit{Example $2$.} Let us consider a multipoint boundary-value problem for the differential system \eqref{1.6.1t1} with $A(t) \equiv O_{m}$. The boundary conditions contain derivatives of integer and/or \textit{fractional} orders (in the sense of Caputo \cite{Kilbas2006}) at certain points $t_k\in[a, b]$, $k=0,\ldots,N$. These conditions become
\begin{equation*}
By:=\sum _{k=0}^N \sum _{j=0}^s \alpha_{k,j} \big(^{\textit{C}}\textmd{D} _{a+}^{\beta_{k,j}}y\big)(t_{k})=c.
\end{equation*}
Here, all $\alpha_{k,j} \in \mathbb{C}^{l \times m}$, whereas the nonnegative numbers $\beta_{k,j}$ satisfy
\begin{equation*}
\beta_{k,0}=0 \quad \mbox{whenever} \quad k \in \{1, 2, \ldots, N\}. 
\end{equation*}

Theorem \ref{th_fredh high} asserts that index of the operator $(L,B)$ equals $m-l$. Let us find its Fredholm numbers. Since $Y(\cdot)=I_{m}$,  the characteristic matrix is of the form
\begin{equation*}
M(L, B)=\sum _{k=0}^N \sum _{j=0}^s \alpha_{k,j} \big(^{\textit{C}}\textmd{D} _{a+}^{\beta_{k,j}}I_{m}\big)=\sum _{k=0}^N \alpha_{k,0}
\end{equation*}
because the derivatives $\big(^{\textit{C}}\textmd{D} _{a+}^{\beta_{k,j}}I_{m}\big)=0$ whenever $\beta_{k,j} >0$. Hence, by Theorem~\ref{th dimker}, we conclude that
\begin{equation*}
\operatorname{dim} \operatorname{ker}(L,B)=\operatorname{dim} \operatorname{ker}\left(\,\sum_{k=0}^{N}\alpha_{k,0}\right)=m -
\operatorname{rank}\left(\,\sum_{k=0}^{N}\alpha_{k,0}\right)
\end{equation*}
and
\begin{align*}
\operatorname{dim} \operatorname{coker}(L,B)&= - m+l + \operatorname{dim} \operatorname{coker}\left(\,\sum_{k=0}^{N}\alpha_{k,0} \right)\\
&=l-\operatorname{rank}\left(\,\sum_{k=0}^{N}\alpha_{k,0} \right).
\end{align*}

These formulas show that the Fredholm numbers of the problem do not depend on the length of the interval $(a, b)$ and on the choice of the points $\{t_k\}_{k=0}^N \subset [a, b]$ and matrices $\alpha_{k,j}$ with $j \geq 1$.

\textit{Example $3$.} Consider a two-point boundary-value problem for a system of second-order differential equations generated by the expression
\begin{equation*}
    (Ly)(t):= y^{\prime \prime} (t)+Ay^{\prime}(t),\quad
t \in(a,b),
\end{equation*}
and the boundary operator
\begin{equation*}
By:=\sum^{n+1} _{k=0} \left(\alpha_{k} y^{(k)}(a)+ \beta_{k}y^{(k)}(b)\right).
\end{equation*}
Here, $A\in\mathbb{C}^{m\times m}$, and all $\alpha_{k},\beta_{k}\in \mathbb{C}^{l \times m}$. Thus, we have the bounded operator
$$
(L,B)\colon (W^{n+2}_p)^m\rightarrow (W^{n}_p)^m\times\mathbb{C}^l
$$
and characteristic matrix $M(L,B) \in \mathbb{C}^{l\times 2m}$.

It is easy to verify that
$$
Y_1(t)\equiv I_m \quad\mbox{and}\quad Y_2(t)=\varphi (A, t),
$$
where the function $\varphi (\lambda, t):=1-\exp(-\lambda(t-a))\lambda^{-1}$ is an entire analytic function of $\lambda \in \mathbb{C}$ for each fixed $t \in [a, b]$. Then
$$
[BY_1]=\sum_{k=0}^{n+1} \left(\alpha_kI_m^{(k)}(a)+\beta_kI_m^{(k)}(b)\right)=(\alpha_0+\beta_0)I_m
$$
and
$$
[BY_2]=\sum_{k=0}^{n+1} \left(\alpha_k\varphi^{(k)}(A, a)+\beta_k\varphi^{(k)}(A, b)\right).
$$
However
$$
Y_2^{(k)}(t)=(-1)^kA^k\exp (-A(t-a))
$$
whenever $k \in \{0, \ldots, n+1\}$. Hence,
$$
[BY_2]=\sum_{k=0}^{n+1} \left(\alpha_kI_m+\beta_k\exp (-A(b-a))\right)(-A)^k.
$$
The characteristic block matrix thus becomes
$$
M(L, B)= \left( \alpha_0+\beta_0; \sum_{k=0}^{n+1} \left(\alpha_k+\beta_k\exp (-A(b-a))\right)(-A)^k\right).
$$

According to Theorem \ref{th dimker}, the dimensions of the kernel and co-kernel of the inhomogeneous boundary-value problem respectively equals the dimensions of the kernel and co-kernel of the matrix $M(L, B)$.
Specifically, if $\beta_k \equiv 0$ and if the problem is one-point, then $$
M(L, B)= \left( \alpha_0; \sum_{k=0}^{n+1} \alpha_k(-A)^k\right).
$$
We see in this case that the Fredholm numbers of the boundary-value problem do not depend on the length of the interval $(a,b)$.

Note that the matrix $\exp (-A(b-a))$ can be found in an explicit form because every entire analytic function of a number matrix $A \in \mathbb{C}^{m \times m}$ is a polynomial of $A$. This polynomial is expressed via the matrix $A$ by the Lagrange--Sylvester Interpolation Formula (see, e.g., \cite{Gan1959}). Its degree is no greater than $m-1$.

\textit{Example $4$.} Consider a two-point boundary-value problem for another system of second-order differential equations
\begin{equation*}
    (Ly)(t):= y^{\prime \prime} (t)+Ay(t),\quad
t \in(a,b),
\end{equation*}
where $A\in\mathbb{C}^{m\times m}$. The boundary conditions induced by the same operator as that in Example 3; namely
\begin{equation*}
By=\sum^{n+1} _{k=0} \left(\alpha_{k} y^{(k)}(a)+ \beta_{k}y^{(k)}(b)\right).
\end{equation*}

It is easy to check in this case that, for each fixed $t \in [a, b]$, the fundamental matrix-valued functions $Y_{1}(t)$ and $Y_{2}(t)$ are entire functions of the numerical matrix $A$ given by some convergent power series. Then
\begin{align*}
[BY_1]&=\sum_{\substack{k=1\\k \,\, is \,\, odd}}^{n+1} \beta_k(-1)^kA^{k}\\
&+\sum_{\substack{k=0\\k \,\, is \,\, even}}^{n+1} \alpha_k(-1)^k(\sqrt{A})^{2k-1}\sin \big(\sqrt{A}(b-a)\big)\\
&+\sum_{\substack{k=1\\k \,\, is \,\, odd}} \beta_k(-1)^kA^{k}\cos \big(\sqrt{A}(b-a)\big)
\end{align*}
and
\begin{align*}
[BY_2]&=\sum_{\substack{k=1\\k \,\, is \,\, even}}^{n+1} \alpha_k(-1)^kA^{k}\\
&+\sum_{\substack{k=0\\k \,\, is \,\, even}}^{n+1} \alpha_k(-1)^kA^{k}\cos \big(\sqrt{A}(b-a)\big)\\
&+\sum_{\substack{k=1\\k \,\, is \,\, odd}}^{n+1} \beta_k(-1)^k(\sqrt{A})^{2k-1}\sin \big( \sqrt{A}(b-a)\big),
\end{align*}
with the block characteristic matrix $M(L,B)=[BY_1; BY_2]$.

Specifically, if $\beta_k \equiv 0$ (the case of  one-point boundary-value problem), then
\begin{gather*}
M(L,B) = \left[\,\sum_{\substack{k=0\\k \,\, is \,\, even}}^{n+1} \alpha_k(-1)^k(\sqrt{A})^{2k-1}\sin \big(\sqrt{A}(b-a)\big); \right. \\
\left.\sum_{\substack{k=1\\k \,\, is \,\, even}}^{n+1} \alpha_k(-1)^kA^{k}+
\sum_{\substack{k=0\\k \,\, is \,\, even}}^{n+1}
\alpha_k(-1)^kA^{k}\cos \big(\sqrt{A}(b-a)\big)  \right].
\end{gather*}
Unlike Example 3, this matrix depends in general on the length of the interval $(a,b)$.

If $\alpha_k \equiv 0$, $k$ is even, $\beta_k \equiv 0$, and $k$ is odd, then the characteristic matrix $M(L,B)=O_{2m \times l}$. Therefore, its Fredholm numbers take the largest possible values.

As in Example 3, the matrices $\sin \big(\sqrt{A}(b-a)\big)$ and $\cos \big( \sqrt{A}(b-a)\big)$ can be exactly found as Lagrange--Sylvester interpolation polynomials.

\textit{Example 5.} Consider the following linear boundary-value problem for a system of $m$ first-order differential equations:
\begin{equation}\label{1.6.1t}
    Ly(t):= y'(t)=f(t),\quad
t \in(a,b), \qquad By= c,
\end{equation}
where $f(\cdot) \in (W_{p}^{n})^{m}$ and $c \in \mathbb{C}^{l}$ and $B$
is an arbitrary linear continuous operator from
$(W_{p}^{n+1})^{m}$ to $\mathbb{C}^{l}$. We suppose that $1\leq p<\infty$.

Note that $Y(\cdot)=I_{m}$ is the unique solution of the linear homogeneous matrix equation of the form \eqref{1.6.1t} with the initial Cauchy condition
\begin{equation*}\label{r3}
   Y'(t)=0,\quad t\in (a,b), \qquad Y(a)=I_{m}.
  \end{equation*}
According to \eqref{st anal}, we have
\begin{equation*}
M(L,B)=[BY]=\sum _{i=0}^{n} \alpha_{i} Y^{(i)}(a)+
\int_{a}^b \Phi(t)Y^{(n+1)}(t){\rm d}t=\alpha_{0}.
\end{equation*}
Therefore,
\begin{equation*}
\operatorname{dim} \operatorname{ker}(M(L,B))=\operatorname{dim} \operatorname{ker}(\alpha_{0})
\end{equation*}
and
\begin{equation*}
\operatorname{dim} \operatorname{coker}(M(L,B))=\operatorname{dim} \operatorname{coker}(\alpha_{0}).
\end{equation*}
Hence, the boundary-value problem \eqref{1.6.1t} is well posed if and only if the number matrix $\alpha_{0}$ is square and nonsingular.

\section{Proof of the index theorem}\label{section5}

We previously establish an auxiliary result concerning the Cauchy problem for the differential system~\eqref{bound_pr_1}. We introduce the linear operator
\begin{equation}\label{operC}
C\colon (W^{n+r}_p)^m\rightarrow\mathbb{C}^{rm}
\end{equation}
by putting
\begin{equation*}
Cy := \operatorname{col}\left(y(a),y'(a),\ldots,y^{(r-1)}(a)\right) \quad\mbox{for any} \quad y\in (W^{n+r}_p)^m.
\end{equation*}
It follows from the continuous embedding $W^{n+r}_p[a,b]\subset C^{r-1}[a,b]$ that the operator $C$ is well defined and bounded.

\begin{lemma}\label{ob Koshi}
The linear bounded operator $(L,C)$ is invertible on the a pair of Banach spaces
\begin{equation}\label{operCL}
(L,C)\colon (W^{n+r}_p)^m\rightarrow(W^{n}_p)^m\times\mathbb{C}^{rm}.
\end{equation}
\end{lemma}

\begin{proof} We first treat the $r=1$ case by the mathematical induction in $n\in\mathbb{N}\cup \{0\}$. Thus, we consider the Cauchy problem
\begin{align}\label{zad Koshi}
    Ly(t):= y'(t)+A(t)y(t)&=f(t),\quad
t \in(a,b),\\
\label{zad Koshi2}
    y(a)&= c.
\end{align}
Here, $A(\cdot) \in (W_p^n)^{m\times m}$, $f(\cdot) \in (W^n_p)^m$, $c\in \mathbb{C}^{m}$, and $y(\cdot)\in (W_{p}^{n+1})^m$, with $(L,C)$ being a bounded linear operator between the spaces
\begin{equation}\label{L,B1}
  (L,C)\colon(W_{p}^{n+1})^{m} \rightarrow(W^{n}_p)^m\times\mathbb{C}^{rm}.
\end{equation}
Each solution to the problem \eqref{zad Koshi}, \eqref{zad Koshi2} becomes
\begin{equation}\label{rozv}
y(t)=Y(t)c+Y(t)\int_{a}^t Y^{-1}(s)f(s){\rm d}s.
\end{equation}

Let $n=0$. Since the homogeneous Cauchy problem has a unique solution $y=0$, the operator $(L,C)$ is one-to-one. It follows from the assumption $A(\cdot) \in (L_p)^{m\times m}$ that $Y(\cdot) \in (W^1_p)^{m\times m}$. Since $W^1_p$ is a Banach algebra, the inverse $Y^{-1}(\cdot)$ also belongs to $(W^1_p)^{m\times m}$, which implies that the right-hand side of equality \eqref{rozv} belongs to $(W^1_p)^m$. Thus, the Cauchy problem \eqref{zad Koshi}, \eqref{zad Koshi2} has a solution $y(\cdot)$ of class
$(W^1_p)^m$ whatever $f(\cdot) \in (W^n_p)^m$ and $c\in \mathbb{C}^{m}$; i.e., the operator \eqref{L,B1} is onto for $n=0$.

Assume now that the conclusion of the lemma is true for a certain number $n=k\in \mathbb{N} \cup \{0\}$. Let us prove that the conclusion holds true for $n=k+1$. We use the same reasoning to show that the operator \eqref{L,B1} is one-to-one. It remains to show that this operator is onto. By the inductive assumption, the matrix-valued function $Y(\cdot)$ belong to $(W^{k+1}_p)^{m\times m}$ as a solution to the matrix Cauchy problem
\begin{equation*}
Y'(t)+A(t)Y(t)=O_{m},\quad t\in(a,b),\qquad Y(a)=I_{m}.
\end{equation*}
Hence, $Y'=-AY\in(W^{k+1}_p)^{m\times m}$ because $W^{k+1}_p$ is a Banach algebra. This implies that $Y(\cdot)\in(W^{k+2}_p)^{m\times m}$. Hence, $Y^{-1}(\cdot)\in(W^{k+2}_p)^{m\times m}$ because $W^{k+2}_p$ is a Banach algebra too. Thus, the Cauchy problem \eqref{zad Koshi}, \eqref{zad Koshi2} has a solution $y(\cdot)\in(W^{k+2}_p)^m$ of the form \eqref{rozv} whatever $f(\cdot) \in (W^n_p)^m$ and $c\in \mathbb{C}^{m}$. Therefore, the operator \eqref{L,B1} is onto  for $n=k+1$. We then conclude that this continuous bijection operator is an isomorphism by the bounded inverse theorem.

Let us now prove the lemma in the $r\geq2$ case. Consider an  inhomogeneous Cauchy problem
\begin{gather}
Ly(t)=f(t), \quad t\in(a,b), \label{neodnKoshi} \\
y^{(j-1)}(a)=c_j, \quad j\in \{1,\dots, r\}. \label{neodnKoshi2}
\end{gather}
Here, the vectors $c_j \in \mathbb{C}^{m}$ are arbitrarily given. The problem \eqref{neodnKoshi}, \eqref{neodnKoshi2} is reduced to the following Cauchy problem for a system of first-order differential equations:
\begin{equation}\label{zad Ko 1}
x'(t)+K(t)x(t)=g(t), \quad t\in(a,b),\qquad
x(a)=c.
\end{equation}
Here,
\begin{gather*}
x(\cdot):=\mathrm{col}\big(y(\cdot), y'(\cdot), \dots, y^{(r-1)}(\cdot)\big)\in (W^{n+r}_p)^{rm},\\
g(\cdot):=\mathrm{col}\big(\underbrace{0, \dots, 0}_{(r-1)m}, f(\cdot)\big)\in (W^{n}_p)^{rm},\\
c:=\mathrm{col}\big(c_1, \dots, c_r\big)\in \mathbb{C}^{rm},
\end{gather*}
and the block matrix-valued function $K(\cdot) \in (W^{n}_p)^{rm\times rm}$ is defined by the formula
\begin{equation}\label{AA-reduced}
K(\cdot):=\left(
\begin{array}{ccccc}
O_m & -I_m & O_m & \ldots & O_m \\
O_m & O_m & -I_m & \ldots & O_m \\
\vdots & \vdots & \vdots & \ddots & \vdots \\
O_m & O_m & O_m & \ldots & -I_m \\
A_0(\cdot) & A_1(\cdot) & A_2(\cdot) & \ldots & A_{r-1}(\cdot)\\
\end{array}\right).
\end{equation}
We have just proved that the Cauchy problem \eqref{zad Ko 1} has a unique solution $x(\cdot)\in (W^{n+r}_p)^{rm}$ for arbitrary $g\in (W^{n}_p)^{rm}$ and $c\in C^{rm}$. Hence, the linear bounded operator \eqref{operCL} is bijective. Therefore, it is an isomorphism by the bounded inverse theorem.
\end{proof}

\begin{proof}[Proof of Theorem \ref{th_fredh high}.] We separately treat three cases: $l=rm$, $l>rm$, and $l<rm$.

The $l=rm$ case. According to Lemma~\ref{ob Koshi}, we have the isomorphism
\begin{equation}\label{isom(L,C)}
(L,C)\colon (W^{n+r}_p)^m\leftrightarrow(W^{n}_p)^m\times\mathbb{C}^{rm}.
\end{equation}
Note that
\begin{equation*}
(L,B)=(L,C)+(0,B-C)
\end{equation*}
and that $(0,B-C_{l,m})$ is a finite-dimensional operator.  Hence (see, e.g., \cite[Chapter~IV, Section~5, Subsection~2]{Kato_book}), the operator $(L,B)$ is Fredholm between the spaces \eqref{isom(L,C)}, and its index coincides with the zero index of the operator \eqref{isom(L,C)}, which is what was to be proved in this case.

The $l>rm$ case. Put
\begin{equation*}
C_{l-rm}\,y:=\operatorname{col}\bigl(Cy,\underbrace{0,\dots ,0}_{l-rm}\bigr)
\quad\mbox{for every}\quad y(\cdot)\in (W_{p}^{n+r})^m.
\end{equation*}
It follows directly from Lemma~\ref{ob Koshi} that the bounded operator
\begin{equation*}
(L,C_{l-mr}):(W^{n+r}_p)^m\to(W^{n}_p)^m\times\mathbb{C}^{l}
\end{equation*}
has zero kernel and the closed domain
\begin{equation*}
(W^{n}_p)^m\times\mathbb{C}^{rm}\times\{0\}^{l-rm}.
\end{equation*}
Hence, this operator is Fredholm with index $rm-l$; then so is the operator \eqref{operCL} because
\begin{equation*}
(L,B)=(L,C_{l-mr})+(0,B-C_{l-mr})
\end{equation*}
and since $(0,B-C_{l-mr})$ is a finite-dimensional operator.

The $l<rm$ case. We introduce the bounded linear operator \begin{equation*}
P_{rm-l}:\mathbb{C}^{rm}\to\mathbb{C}^{l}
\end{equation*}
by the formula
\begin{equation*}
P_{rm-l}\operatorname{col}(c_{1},\ldots,c_{l},\ldots,c_{rm})=
\operatorname{col}(c_{1},\ldots,c_{l})
\end{equation*}
for every $c_{1},\ldots,c_{rm}\in\mathbb{C}$. It follows directly from Lemma~\ref{ob Koshi} that the bounded operator
\begin{equation}\label{LP-trankated}
(L,P_{rm-l}C):(W^{n+r}_p)^m\to(W^{n}_p)^m\times\mathbb{C}^{l}
\end{equation}
is onto and that the kernel of this operator equals $(L,C)^{-1}(\mathcal{O}\times\mathbb{C}^{rm-l})$
and therefore is of dimension $rm-l$; here, $\mathcal{O}$ stands for the null subspace of $(W^{n}_p)^m\times\mathbb{C}^{l}$. Hence, the operator \eqref{LP-trankated} is  Fredholm with index $rm-l$; then so is the operator \eqref{operCL} because
\begin{equation*}
(L,B)=(L,P_{rm-l}C)+(0,B-P_{rm-l}C)
\end{equation*}
and since $(0,B-P_{rm-l}C)$ is a finite-dimensional operator.
\end{proof}

\section{Proof of the Fredholm numbers theorem}\label{section6}

\begin{proof}[Proof of Theorem \ref{th dimker}.] According to Lemma~\ref{operCL}, the restriction of the mapping \eqref{operC} to the kernel of $L$ sets an isomorphism
\begin{equation}\label{L-isom}
C:\ker L\leftrightarrow \mathbb{C}^{rm}.
\end{equation}
The inverse of \eqref{L-isom} is defined for every vector
\begin{equation*}
q:=\operatorname{col}({q_{1},\ldots,q_{r}})\in\mathbb{C}^{rm},
\quad\mbox{with}\quad q_{1},\ldots,q_{r}\in\mathbb{C}^{m},
\end{equation*}
as follows:
\begin{equation*}
q\mapsto y(\cdot):=\sum_{i=1}^{r}Y_i(\cdot)q_{i}.
\end{equation*}
Recall that $Y_i(\cdot)$ is a unique solution to the matrix Cauchy problem \eqref{zad kosh1}, \eqref{zad kosh2}. Then
\begin{align*}
By&=B\left(\,\sum_{i=1}^{r}Y_i(\cdot)q_{i}\right)=
\sum_{i=1}^{r}B(Y_i(\cdot)q_{i})=\sum_{i=1}^{r}[BY_i]q_{i}\\
&=M(L,B)q.
\end{align*}
The third equality is due to Lemma~\ref{dija oper priam} given at the end of this section. Hence,
\begin{equation*}
C\bigl(\ker(L,B)\bigr)=\ker M(L,B),
\end{equation*}
which yields the first required formula \eqref{dimker}.

The second formula \eqref{dimcoker1} follows from \eqref{dimker} and Theorem~\ref{th_fredh high}, namely:
\begin{align*}
\operatorname{dim}\operatorname{coker}(L,B)&=
\operatorname{dim}\operatorname{ker}(L,B)-
\operatorname{dim}\operatorname{ind}(L,B)\\
&=\operatorname{dim}\operatorname{ker}M(L,B)-(rm-l)\\
&=\operatorname{dim}\operatorname{coker}M(L,B).
\end{align*}
As to the last equality, recall that the dimension of the index of any number $l\times rm$ matrix is $rm-l$.
\end{proof}

\begin{lemma}\label{dija oper priam}
Let $E$ be a linear space. The equality
\begin{equation*}\label{rivn_matruz}
\mathcal{B}(\Upsilon q) = [\mathcal{B}\Upsilon]q
\end{equation*}
holds true for an arbitrary linear operator $\mathcal{B}\colon E^{m}\to\mathbb{C}^l$, matrix $\Upsilon\in E^{m\times\mu}$, and vector $q\in\mathbb{C}^\mu$, with $m,l,\mu\in\mathbb{N}$.
\end{lemma}

Here, as similar to the $E=W^{n+r}_{p}$ case, $[\mathcal{B}\Upsilon]$ stands for  the number $l\times\mu$ matrix whose $j$-th column is the action of $\mathcal{B}$ on the $j$-th column of~$\Upsilon$.

\begin{proof} It is similar to that given in \cite[Lemma~6]{AtlMikh2018} in the case where $m=\mu=l$ and when $E$ is a Sobolev space. We will give the proof for the reader's convenience. Letting $i\in \{1,2,\ldots,m\}$ and $k\in \{1,2,\ldots,\mu\}$ and $j\in \{1,2,\ldots, l\}$, we write the matrix-valued function $\Upsilon$ and column vector $q$ in the form $\Upsilon=(\upsilon_{i,k}(\cdot))$ and $q=\mathrm{col}(q_{k})$. Put $\mathrm{col}(\alpha_{j}):=[\mathcal{B}\Upsilon]q$ and $(c_{j,k}):=[\mathcal{B}\Upsilon]$. Since
$$
\alpha_{j}=\sum_{k=1}^{\mu} c_{j,k}\,q_{k},
$$
we obtain
\begin{align*}
\mathcal{B}(\Upsilon q)&=
\mathcal{B}\,\mathrm{col}\left
(\,\sum_{k=1}^{\mu}\upsilon_{i,k}(\cdot)q_{k}\right)_{i=1}^{m}=
\sum_{k=1}^{\mu}q_{k}
\mathcal{B}\,\mathrm{col}(\upsilon_{i,k}(\cdot))_{i=1}^{m}\\
&=\sum _{k=1}^\mu q_{k}\,\mathrm{col}\left(c_{j,k}\right)_{j=1}^{l}=
\mathrm{col}\left(\sum_{k=1}^\mu q_{k}\,c_{j,k} \right)_{j=1}^{l}=
\mathrm{col}(\alpha_{j})_{j=1}^{l}\\
&=[\mathcal{B}\Upsilon]q,
\end{align*}
as was to be proved.
\end{proof}

\section{Proof of the limit theorems}\label{section7}

To prove Theorem~\ref{koef matr}, we need some auxiliary results.

\begin{lemma}\label{ekviv sul}
The following three conditions are equivalent to each other as $k \rightarrow \infty$:
\begin{itemize}
\item [(I)] $L(k) \rightarrow L$ in the strong operator topology;
\item [(II)] $L(k) \rightarrow L$ in the uniform operator topology;
\item [(III)] $A_{r-j}(\cdot,k) \rightarrow A_{r-j}(\cdot)$ in the Banach space $(W^{n}_p)^{m\times m}$ for each $j\in\{1,\ldots, r\}$.
\end{itemize}
\end{lemma}

Here, of course, $L$ and $L(k)$ are considered as bounded operators from  $(W^{n+r}_p)^{m}$ to $(W^{n}_p)^{m}$.

\begin{proof} It suffices to show that the following implications hold:
$$
(\mathrm{III})\Longrightarrow (\mathrm{II}) \Longrightarrow (\mathrm{I}) \Longrightarrow (\mathrm{III}).
$$

Let us prove that $(\mathrm{III})\Rightarrow(\mathrm{II})$. Assume  $(\mathrm{III})$ to be valid. We first consider the $n\in\mathbb{N}$ case. Given $y\in (W^{n+r}_p)^{m}$, we have
\begin{align*}
\bigl\|(L(k)-L)y\bigr\|_{n,p}&\leq
\sum_{j=1}^r\bigl\|(A_{r-j}(\cdot,k)- A_{r-j}(\cdot))y\bigr\|_{n,p} \leq \\
&\leq c_{n,p}\sum_{j=1}^r\bigl\|(A_{r-j}(\cdot,k)- A_{r-j}(\cdot))\bigr\|_{n,p}\,\|y\|_{n,p}\\
&\leq c_{n,p}\sum_{j=1}^r\bigl\|(A_{r-j}(\cdot,k)- A_{r-j}(\cdot))\bigr\|_{n,p}\,\|y\|_{n+r,p}
\end{align*}
because $W^{n}_p$ is a Banach algebra; here, the positive number $c_{n,p}$ does not depend on $y$. Hence, the norm of the operator $L(k)-L$ satisfies
\begin{equation*}
\|L(k)-L\|\leq c_{n,p}\sum\limits_{j=1}^r \left\|A_{r-j}(\cdot,k)- A_{r-j}(\cdot)\right\|_{n,p}\to0.
\end{equation*}
If $n=0$, then for every $y\in (W^{r}_p)^{m}$ we obtain
\begin{align*}
\bigl\|(L(k)-L)y\bigr\|_{0,p}&\leq
\sum_{j=1}^r\bigl\|(A_{r-j}(\cdot,k)- A_{r-j}(\cdot))\bigr\|_{0,p}\,\|y\|_{\infty}\\
&\leq c_{r,p}'\sum_{j=1}^r\bigl\|(A_{r-j}(\cdot,k)- A_{r-j}(\cdot))\bigr\|_{0,p}\,\|y\|_{r,p};
\end{align*}
here, $c_{r,p}'$ is the norm of the bounded embedding operator $W^{r}_p\hookrightarrow L_{\infty}$. Hence, $\|L(k)-L\|\to0$ in the $n=0$ case as well. The implication $(\mathrm{III})\Rightarrow(\mathrm{II})$ is substantiated.

The implication $(\mathrm{II})\Rightarrow(\mathrm{I})$ is trivial.

Let us prove that $(\mathrm{I})\Rightarrow(\mathrm{III})$. Suppose  $(\mathrm{I})$ to be valid. Then, whatever $Y(\cdot)\in(W^{n+r}_p)^{m\times m}$, we have $L(k)Y\to LY$, i.e.
\begin{equation}\label{vusok p}
\sum\limits_{j=1}^r A_{r-j}(\cdot,k)Y^{(r-j)}(\cdot)\to
\sum\limits_{j=1}^r A_{r-j}(\cdot)Y^{(r-j)}(\cdot);
\end{equation}
in this paragraph, each convergence is considered in $(W^{n}_p)^{m\times m}$. Putting $Y(t):=I_m$ here, we get $A_{0}(\cdot,k)\to A_{0}(\cdot)$. Then, letting $Y(t):=tI_m$ in \eqref{vusok p}, we arrive at
\begin{equation*}
A_{1}(t,k)+A_{0}(t,k)t\to A_{1}(t)+A_{0}(t)t,
\end{equation*}
which implies that $A_{1}(\cdot,k)\to A_{1}(\cdot)$. Further, putting $Y(t):=t^{2}I_m$ in \eqref{vusok p}, we arrive at
\begin{equation*}
2A_{2}(t,k)+2A_{0}(t,k)t+A_{0}(t,k)t^{2}\to 2A_{2}(t)+2A_{0}(t)t+A_{0}(t)t^{2}
\end{equation*}
which yields $A_{2}(\cdot,k)\to A_{2}(\cdot)$. Continuing this process sequentially up to the $r$th step, we prove $(\mathrm{III})$. The implication $(\mathrm{I})\Rightarrow(\mathrm{III})$ is also substantiated.
\end{proof}

Given $k\in\mathbb{N}$ and $i\in\{1,\dots, r\}$, we let $Y_i(\cdot,k)$ denote a unique solution to the matrix Cauchy problem
\begin{align*}
Y_i^{(r)}(t,k)+\sum\limits_{j=1}^rA_{r-j}(t,k)Y_i^{(r-j)}(t,k)&=O_{m},\quad t\in (a,b),\\
Y_i^{(j-1)}(a,k) &= \delta_{i,j}I_m,\quad j \in \{1,\dots, r\}. 
\end{align*}
Note that $Y_i(\cdot,k)\in(W^{n+r}_p)^{m\times m}$ due to Lemma~\ref{ob Koshi}. Recall that $Y_i(\cdot)\in(W^{n+r}_p)^{m\times m}$ is a unique solution to the matrix Cauchy problem \eqref{zad kosh1}, \eqref{zad kosh2}.

\begin{lemma}\label{ekviv rivnom}
Suppose that condition $(\mathrm{III})$ of Lemma \ref{ekviv sul} is satisfied. Then
\begin{equation}\label{rr77}
Y_i(\cdot,k) \rightarrow Y_i(\cdot) \quad\mbox{in}\quad (W^{n+r}_p)^{m\times m}
\end{equation}
for each $i\in\{1,\dots, r\}$.
\end{lemma}

The proof of this lemma is based on the following result:

Given $\mu\in\mathbb{N}$, we let $\mathcal{Y}_{p,\mu}^{n+1}$ denote the set of all matrix-valued functions $\mathcal{X}(\cdot)\in (W_{p}^{n+1})^{\mu\times \mu}$ such that $\mathcal{X}(a)=I_{\mu}$ and that $\det\mathcal{X}(t)\neq 0$ for every $t\in[a,b]$. We endow $\mathcal{Y}_{p,\mu}^{n+1}$ with the metric
$$
d_{n+1,p}(\mathcal{X}(\cdot),\mathcal{Z}(\cdot)):=
\|\mathcal{X}(\cdot)-\mathcal{Z}(\cdot)\|_{n+1,p}.
$$

\begin{lemma}\label{th1}
Given $\mathcal{K}(\cdot)\in (W_p^{n})^{\mu\times \mu}$, we let  $\mathcal{X}(\cdot)$ denote a unique solution to the matrix Cauchy problem
\begin{equation*}\label{rr6}
\mathcal{X}'(t)+\mathcal{K}(t)\mathcal{X}(t)=O_\mu,\quad t\in (a,b), \qquad \mathcal{X}(a)=I_{\mu}.
\end{equation*}
Then the nonlinear mapping
\begin{equation*}
\mathcal{K}(\cdot)\mapsto\mathcal{X}(\cdot)
\end{equation*}
is a homeomorphism between the Banach space $(W_{p}^{n})^{\mu\times \mu}$ and the metric space $\mathcal{Y}_{p,\mu}^{n+1}$.
\end{lemma}

The proof of Lemma \ref{th1} is given in \cite[Theorem~3]{AtlMikh2018}.

\begin{proof}[Proof of Lemma~\ref{ekviv rivnom}.]
If $r=1$, then this lemma is a direct consequence of Lemma~\ref{th1} considered for $\mu=m$.

Let us now treat the $r\geq2$ case. Using the block matrix-valued function \eqref{AA-reduced} of class $(W^{n}_p)^{rm\times rm}$, we reduce the matrix Cauchy problem \eqref{zad kosh1}, \eqref{zad kosh2} to a boundary-value problem for a system of first-order differential equations. Given $i\in \{1,\dots, r\}$, we let
$Z_i(\cdot)$ denote a unique solution to the matrix Cauchy problem
\begin{gather*}
Z'_i(t)+K(t)Z_i(t)=O_{rm\times m},\quad t\in (a,b),\\
Z_i(a) = J_{i}:=\mathrm{col}\bigl(O_m, \dots, O_m, \underbrace{I_m}_{i}, O_m,\dots, O_m\bigr).
\end{gather*}
Here, of course, $O_{rm\times m}$ stands for the zero $rm\times m$ matrix, and the $rm\times m$ matrix $J_{i}$ is defined to consist of $r$ square blocks. Owing to Lemma~\ref{ob Koshi}, the solution $Z_i(\cdot)$ belongs to $(W^{n+1}_p)^{rm\times m}$. We write down it in the form
\begin{equation*}
Z_i(\cdot):=\operatorname{col}(Z_{i,1}(\cdot), \dots, Z_{i,r}(\cdot)),
\end{equation*}
where $Z_{i,j}\in(W^{n+1}_p)^{m\times m}$ whenever $1\leq j\leq r$. We also put
\begin{equation*}
Z(\cdot):=\left(Z_1(\cdot),\ldots,Z_r(\cdot)\right)\in
(W^{n+1}_p)^{rm\times rm}
\end{equation*}
and observe that
\begin{align}
Z'(t)+K(t)Z(t)&=O_{rm},\quad t\in (a,b),
\label{ZZ zad kosh1}\\
Z(a)&= I_{rm}.\label{ZZ_k^(j)(a)}
\end{align}

The following result is known (see, e.g., \cite[Part~II, Section~2.6]{Cartan71}):

\begin{lemma}\label{vus i per}
The solution $Y_i(\cdot)$, with $i\in \{1,\dots, r\}$, of the Cauchy problem \eqref{zad kosh1}, \eqref{zad kosh2} relates to the solution $Z(\cdot)$ of the Cauchy problem \eqref{ZZ zad kosh1}, \eqref{ZZ_k^(j)(a)} by the formula
\begin{equation*}
Y_i^{(j-1)}(\cdot)=Z_{i,j}(\cdot)\quad\mbox{for each}\quad j\in\{1,\ldots,r\}.
\end{equation*}
\end{lemma}

Given $k\in\mathbb{N}$, we let $K(\cdot,k)$ denote the matrix-valued function \eqref{AA-reduced} in which every $A_{r-j}(\cdot)$, with $j=r,\ldots,1$, is replaced with $A_{r-j}(\cdot,k)$. Changing $K(t)$ for $K(t,k)$ in the Cauchy problem \eqref{ZZ zad kosh1}, \eqref{ZZ_k^(j)(a)}, we consider its unique solution $Z(t,k)$ with the corresponding components $Z_{i,j}(t,k)$. By Lemma~\ref{vus i per},
\begin{equation*}
Y_i^{(j-1)}(\cdot,k)=Z_{i,j}(\cdot,k)\quad\mbox{for each}\quad i,j\in\{1,\ldots,r\}.
\end{equation*}
Since condition $(\mathrm{III})$ is satisfied, we have $K(\cdot,k)\to K(\cdot)$ in $(W^{n}_{p})^{rm\times rm}$. Hence, $Z(\cdot,k)\to Z(\cdot)$
in $(W^{n+1}_{p})^{rm\times rm}$ by Lemma~\ref{th1} for $\mu=rm$. This implies by Lemma~\ref{vus i per} that
\begin{equation*}
Y_i^{(j-1)}(\cdot,k)\to Y_i^{(j-1)}(\cdot)
\quad\mbox{in}\quad (W^{n+1}_{p})^{m\times m}
\end{equation*}
for each $i,j\in\{1,\ldots,r\}$, which  entails the required formula~\eqref{rr77}.
\end{proof}

\begin{proof}[Proof of Theorem \ref{koef matr}.] Suppose that $(L(k),B(k))\xrightarrow{s}(L,B)$. Then, by Lem\-ma~\ref{ekviv sul}, condition (III) is satisfied. This implies \eqref{rr77} due to Lemma~\ref{ekviv rivnom}. Therefore, $[B(k)Y_i(k)]\to[B(k)Y(k)]$ for each $i\in\{1,\ldots,k\}$, which yields $M(L(k),B(k))\rightarrow M(L,B)$, as was to be proved.
\end{proof}

\begin{proof}[Proof of Theorem \ref{ker coker}.] We suppose that condition \eqref{zb LB11} is satisfied. Then $M(L(k),B(k))\rightarrow M(L,B)$ due to Theorem~\ref{koef matr}. Put $\varrho:=\operatorname{rank}M(L,B)$ so that there exists a nonzero minor of order $\varrho$ of the matrix $M(L,B)$. Hence, the same minor (of order $\varrho$) of the matrix $M(L(k),B(k))$ is nonzero whenever $k\geq1$. Therefore,
$$
\varrho_{k}:=\operatorname{rank}M(L(k),B(k))\geq \varrho
\quad\mbox{whenever}\quad k\geq1.
$$
Hence,
\begin{equation*}
\dim\ker M(L(k),B(k))=rm-\varrho_{k}\leq rm-\varrho=\dim\ker M(L,B)
\end{equation*}
and
\begin{equation*}
\dim\operatorname{coker}M(L(k),B(k))=l-\varrho_{k}\leq l-\varrho= \dim\operatorname{coker}M(L,B)
\end{equation*}
for all sufficiently large $k$. This implies the required formulas \eqref{ner ker} and \eqref{ner coker} in view of Theorem~\ref{th dimker}.
\end{proof}

\section{Appendix}\label{section8}

Let $E_{1}$ and $E_{2}$ be infinite-dimensional complex or real Banach spaces, and suppose that at least one of them has a Schauder basis. Let  $\mathcal{B}(E_{1},E_{2})$ denote the Banach space of all bounded linear operators from $E_{1}$ to $E_{2}$. Then the set of all finite-dimensional operators of class $\mathcal{B}(E_{1},E_{2})$ is sequentially dense in $\mathcal{B}(E_{1},E_{2})$ in the strong operator topology.

Indeed, suppose $E_{1}$ to have a Schauder basis, and let $P_n$, with $n\in \mathbb{N}$, denote the projector of $E_{1}$ onto the linear span of the first $n$ elements of the basis. Then $P_n \xrightarrow{s} I_1$, with each convergence being considered as $n\to\infty$ in Appendix. Here, of course, $I_1$ stands for the identity operator on $E_{1}$. Therefore, $TP_{n}\xrightarrow{s}T$ for every $T\in\mathcal{B}(E_{1},E_{2})$. Each operator $TP_{n}$ is finite-dimensional because $\dim((TP_{n})(E_{1}))\leq n<\infty$. Moreover, the kernel and co-kernel of $TP_{n}$ are infinite-dimensional so that their dimensions do not depend on $\dim\ker T$ and $\dim\operatorname{coker}T$.

The case where $E_{2}$ has a Schauder basis is similarly considered. Let $Q_n$ be the projector of $E_{2}$ onto the linear span of the first $n$ elements of the basis. Then $Q_{n}T\xrightarrow{s}T$ for every $T\in\mathcal{B}(E_{1},E_{2})$, with each operator $Q_{n}T$ being finite-dimensional. The kernel and co-kernel of $Q_{n}T$ are infinite-dimensional.

\small Institute of Mathematics of the National Academy of Sciences of Ukraine \\
Tereshchenkivska Str. 3, 01024 Kyiv, Ukraine; \\
Institute of Mathematics of the Czech Academy of Sciences, \\
Zitna Str. 25, 115 67 Prague, Czech Republic\\
ORCID: 0000-0002-1332-1562, 
vladimir.mikhailets@gmail.com

\vspace{+0.4cm}

\small Institute of Mathematics of the National Academy of Sciences of Ukraine \\
Tereshchenkivska Str. 3, 01024 Kyiv, Ukraine; \\
Institute of Mathematics of the Czech Academy of Sciences, \\
Zitna Str. 25, 115 67 Prague, Czech Republic\\
ORCID: 0000-0003-0186-3185, hatlasiuk@gmail.com

\end{document}